\documentclass[reqno,10pt]{amsart}

\usepackage{amsfonts,amsmath,amssymb,amsthm,amscd,enumerate,comment}
\usepackage[dvipsnames]{xcolor}
\usepackage{hyperref}

\newtheorem{theorem}{Theorem}[section]
\newtheorem{lemma}{Lemma}[section]
\newtheorem{corollary}{Corollary}[section]
\newtheorem{proposition}{Proposition}[section]
\theoremstyle{remark}

\numberwithin{equation}{section}

\theoremstyle{definition}

\DeclareMathOperator{\divergence}{div}
\newcommand{\Hess}{\nabla^2}
\newcommand{\Lie}{\mathcal L}
\newcommand{\dd}{\mathrm d}

\newcommand{\jinv}{\mathrm{inv}}
\newcommand{\janti}{\mathrm{anti}}

\begin{document}

\title[K\"ahlerity of complete almost-K\"ahler gradient Ricci shrinkers]
{K\"ahlerity of complete almost-K\"ahler gradient shrinking Ricci solitons}

\author[Junming Xie]{Junming Xie}
\address{Department of Mathematics, UCSB, Santa Barbara, CA 93106}
\email{junming-xie@ucsb.edu}


\begin{abstract}
	In this paper, we prove that any complete, compact or noncompact, almost-K\"ahler gradient shrinking Ricci soliton is K\"ahler in arbitrary even dimension. Among other applications, combining our result with the classification of complete gradient shrinking K\"ahler--Ricci solitons in complex dimension two, we obtain a full classification of complete almost-K\"ahler gradient shrinking Ricci solitons in real dimension four.
\end{abstract}

\maketitle

\section{Introduction}

An {\em almost-K\"ahler structure} on $M^{2m}$ is a triple $(g,J,\omega)$, where $g$ is a Riemannian metric, $J$ is an almost complex structure compatible with $g$, that is,
\begin{equation*}
	J^2=-\operatorname{Id},
	\qquad
	g(JX,JY)=g(X,Y)
\end{equation*}
for all tangent vectors $X,Y$, and the associated {\em fundamental two-form}
\begin{equation*}
	\omega(X,Y)=g(JX,Y)
\end{equation*}
is closed. The quadruple $(M^{2m},g,J,\omega)$ is then called an {\em almost-K\"ahler manifold}. It is {\em K\"ahler} precisely when $J$ is integrable, or equivalently, when $\nabla\omega=0$.

In almost-K\"ahler geometry, a natural question is to investigate what geometric conditions can force the integrability of an almost complex structure. Goldberg \cite{Goldberg:69} conjectured that every compact almost-K\"ahler Einstein manifold is necessarily K\"ahler. Sekigawa \cite{Sekigawa:87} proved the conjecture in all dimensions under the additional assumption of nonnegative scalar curvature. In real dimension four, LeBrun \cite{LeBrun:15,LeBrun:23} weakened the Einstein hypothesis by proving that every compact almost-K\"ahler four-manifold with nonnegative scalar curvature and divergence-free self-dual Weyl curvature, $\delta W^+=0$, is K\"ahler with constant scalar curvature. However, the condition $\delta W^+=0$ alone does not force integrability; see, e.g., \cite{Kim:16,Bishop-LeBrun:20}. For surveys and related developments on almost-K\"ahler geometry, see \cite{Apostolov-Draghici:03,Oguro-Sekigawa:05}.

In this paper, we investigate the integrability of almost-K\"ahler structures on gradient shrinking Ricci solitons, which naturally generalize positive Einstein metrics. Recall that a {\em connected complete} Riemannian manifold $(M,g)$ is said to be a {\em gradient shrinking Ricci soliton}, or a {\em gradient Ricci shrinker}, if there exists a smooth function $f$ on $M$ such that the Ricci tensor $Rc$ of the metric $g$ satisfies the equation
\begin{equation}\label{eq:soliton}
	Rc+\Hess f=\lambda g,
\end{equation}
for some constant $\lambda>0$. Here, $\Hess f$ denotes the Hessian of $f$, and $f$ is called a {\em potential function} of the Ricci soliton. By Perelman \cite{Perelman:02}, all compact shrinking Ricci solitons are necessarily gradient ones. Accordingly, throughout the paper, $(M^{2m},g,J,\omega,f)$ denotes a complete almost-K\"ahler gradient shrinking Ricci soliton, with $M$ either compact or noncompact.

The concept of Ricci solitons was introduced by Hamilton \cite{Hamilton:88,Hamilton:95} to study the
formation of singularities in the Ricci flow. In particular, gradient shrinking Ricci solitons are self-similar solutions to Hamilton's Ricci flow and often arise as Type-I singularity models in the Ricci flow, as shown by Naber \cite{Naber:10} and Enders--M\"uller--Topping \cite{Enders-Muller-Topping:11}; see also Cao--Zhang \cite{Cao-Zhang:11}. Therefore, understanding the classification and geometry of gradient shrinking Ricci solitons is crucial to the analysis of Ricci flow singularities and related applications.

Through the work of many authors \cite{Hamilton:88,Ivey:93,Perelman:02,Naber:10,Ni-Wallach:08,Cao-Chen-Zhu:08}, complete gradient shrinking Ricci solitons have been fully classified in dimensions two and three. By contrast, the classification and geometry of general gradient shrinking Ricci solitons in dimensions four and higher remain largely open, and relatively few non-rigid examples are known; see \cite{Cao:06,Cao:10,Chow:23,Chow-Kotschwar-Munteanu:25} for surveys of the subject and further references. One natural approach to this broader problem is to impose additional geometric structure and study the resulting class of gradient shrinking Ricci solitons.

In this spirit, we investigate the K\"ahlerity of complete, compact or noncompact, almost-K\"ahler gradient shrinking Ricci solitons in arbitrary even dimension. Our work is motivated in part by the results of Sekigawa \cite{Sekigawa:87} and LeBrun \cite{LeBrun:15} in the setting of nonnegative scalar curvature, a condition satisfied by complete gradient shrinking Ricci solitons \cite{Chen:09}, and in part by H.-D. Cao's 2012 conjecture that every compact non-Einstein Hermitian shrinking Ricci soliton in real dimension four is K\"ahler. Our main result is the following.

\begin{theorem}\label{thm:main}
	Let $(M^{2m},g, J, \omega, f)$, $m\geq2$, be a complete almost-K\"ahler gradient shrinking Ricci soliton. Then $(M^{2m},g,J,\omega,f)$ is K\"ahler.
\end{theorem}

We note that complete gradient shrinking K\"ahler--Ricci solitons in complex dimension two have been recently classified. In the compact case, it is well known that, apart from del Pezzo surfaces carrying K\"ahler--Einstein metrics, the only compact examples are the $U(2)$-invariant Cao--Koiso shrinker on $\mathbb{CP}^2\#\overline{\mathbb{CP}}^{\,2}$ \cite{Cao:96,Koiso:90} and the toric Wang--Zhu shrinker on $\mathbb{CP}^2\#2\overline{\mathbb{CP}}^{\,2}$ \cite{Wang-Zhu:04}. In the noncompact case, recent work of Conlon--Deruelle--Sun \cite{Conlon-Deruelle-Sun:24}, Cifarelli--Conlon--Deruelle \cite{Cifarelli-Conlon-Deruelle:26}, Bamler--Cifarelli--Conlon--Deruelle \cite{Bamler-Cifarelli-Conlon-Deruelle:24}, and Li--Wang \cite{Li-Wang:26} has yielded a complete classification. In addition to the Gaussian shrinker on $\mathbb C^2$ and the product shrinker on $\mathbb{CP}^1\times\mathbb C$, the remaining complete noncompact examples are the $U(2)$-invariant FIK shrinker constructed by Feldman--Ilmanen--Knopf \cite{Feldman-Ilmanen-Knopf:03} on the blow-up of $\mathbb C^2$ at the origin, and the toric BCCD shrinker constructed by Bamler--Cifarelli--Conlon--Deruelle \cite{Bamler-Cifarelli-Conlon-Deruelle:24} on the one-point blow-up of $\mathbb{CP}^1\times\mathbb C$.

Together with the preceding classification results, Theorem~\ref{thm:main} yields the following full classification of complete almost-K\"ahler gradient shrinking Ricci solitons in real dimension four.

\begin{corollary}
	Let $(M^4,g,J,\omega,f)$ be a four-dimensional complete almost-K\"ahler gradient shrinking Ricci soliton. Then $(M^4,g,J,\omega,f)$ is biholomorphic-isometric to one of the following:
	\begin{itemize}
		\item [(a)] a closed del Pezzo surface with its unique K\"ahler--Einstein or K\"ahler--Ricci shrinker metric;
		
		\item [(b)] the Gaussian shrinker $(\mathbb C^2,g_E)$;
		
		\item [(c)] the FIK shrinker on the blow-up of $\mathbb C^2$ at the origin;
		
		\item [(d)] the standard product shrinker $(\mathbb{CP}^1\times\mathbb C,g_c)$;
		
		\item [(e)] the BCCD shrinker on the one-point blow-up of $\mathbb{CP}^1\times\mathbb C$.
	\end{itemize}
\end{corollary}

More generally, Theorem~\ref{thm:main} shows that the almost-K\"ahler category contains no new complete gradient shrinking Ricci solitons beyond the K\"ahler ones. This reduction leads to several immediate consequences. For example, the uniqueness results of Tian--Zhu \cite{Tian-Zhu:00,Tian-Zhu:02} yield the following.

\begin{corollary}
	Let $(M^{2m},g,J,\omega,f)$, $m\geq2$, be a compact almost-K\"ahler shrinking Ricci soliton. Then $g$ is unique, up to automorphisms, among all normalized almost-K\"ahler shrinking Ricci soliton metrics compatible with $J$.
\end{corollary}

Next, combining Theorem~\ref{thm:main} with results of Munteanu--Wang \cite[Theorem~0.1]{Munteanu-Wang:15}, Sun--Zhang \cite[Proposition~3.10]{Sun-Zhang:24}, and Esparza \cite[Corollary~1.3]{Esparza:25} leads to the following topological consequences.

\begin{corollary}
	Let $(M^{2m},g,J,\omega,f)$, $m\geq2$, be a complete almost-K\"ahler gradient shrinking Ricci soliton. Then $M$ is simply connected. Moreover, if $M$ is noncompact, then it has exactly one end.
\end{corollary}

Finally, Theorem~\ref{thm:main}, together with the algebraicity theorem of Sun--Zhang \cite[Theorem~1.1]{Sun-Zhang:24}, implies the following.

\begin{corollary}
	A complete almost-K\"ahler gradient shrinking Ricci soliton is naturally a quasi-projective variety.
\end{corollary}

Many further consequences follow from Theorem~\ref{thm:main} and existing results on shrinking K\"ahler--Ricci solitons; see, for example, \cite{Cao-Sesum:07,Chen-Zhu:12,Munteanu-Wang:14,Li-Ni:20,Dervan-Szekelyhidi:20,Guo-Phong-Song-Sturm:22,He-Ou:24,He-Ou:25,Tran:25,Li:25,Esparza:25b,Cifarelli-Esparza:25,Li-Zhang:26,Xu-Zhang:26,Zhang:26}.

\vspace{0.05in}
\noindent\textbf{Organization of the paper and proof outline.} 
Section~\ref{sec:pre} introduces the notation and conventions used throughout the paper and collects several useful facts needed for the main arguments. 
Section~\ref{sec:coercive} establishes two $e^{-f/2}$-weighted integral identities associated with the exact two-form $\Lie_{\nabla f}\omega=\dd (J\nabla f)^\flat$. Their linear combination is coercive and shows that the Ricci tensor is $J$-invariant, $\nabla_{\nabla f}\omega=0$, and $J\nabla f$ is Killing. 
Section~\ref{sec:torsion} then uses these rigidity properties to prove that the intrinsic torsion $\nabla\omega$ vanishes whenever any one of its arguments lies in $\operatorname{span}\{\nabla f,J\nabla f\}$. 
Finally, Section~\ref{sec:weighted-Sekigawa} combines this transversality with the pointwise identity of Apostolov--Dr\u{a}ghici--Moroianu \cite[Proposition~2.1]{Apostolov-Draghici-Moroianu:01}, which recovers Sekigawa's formula \cite{Sekigawa:87} upon integration in the compact case, to derive a weighted divergence identity. Integrating against suitable cutoffs shows that the weighted integral of the resulting nonnegative terms, including $2\lambda|\nabla\omega|^2$, vanishes. Since $\lambda>0$, this forces $\nabla\omega=0$, proving Theorem~\ref{thm:main}.

\vspace{0.05in}
\noindent\textbf{Acknowledgements.}
The author is grateful to Huai-Dong Cao for valuable suggestions, continued support and encouragement; to Xiaochun Rong and Guofang Wei for their continued support and encouragement; to Ovidiu Munteanu for helpful comments; and to OpenAI for providing twelve months of complimentary access to a dedicated ChatGPT workspace through the \emph{ChatGPT for Academic Researchers} program.

\vspace{0.05in}
\noindent\textbf{AI disclosure statement.}
The author acknowledges the use of ChatGPT (OpenAI) as an interactive aid in preparing this manuscript, including assistance with exposition and with exploring and checking some of the mathematical computations and arguments. All interactions with ChatGPT were initiated and guided by the author, who substantially revised any AI-generated text incorporated into the manuscript and independently verified all mathematical statements and proofs. The author takes full responsibility for the final content of the manuscript.

\section{Preliminaries}\label{sec:pre}

In this section, we fix the notation and conventions used throughout the
paper and recall several basic facts and known results used later. Unless otherwise stated, let $(M^{2m},g,J,\omega)$, $m\geq2$, be an almost-K\"ahler manifold, with $\omega(X,Y)=g(JX,Y)$.

\subsection{Notation and conventions}

We denote by $Rm$, $Rc$, and $R$ the Riemann curvature tensor, Ricci tensor, and scalar curvature of $g$, respectively. All inner products and norms are induced by $g$, and $\dd\mu_g$ denotes the Riemannian volume measure. We use $\flat$ and $\sharp$ for the musical isomorphisms, and
$\iota_X$ and $\Lie_X$ for interior multiplication and the Lie derivative with respect to $X$, respectively.

For two-forms $\alpha,\gamma$ and symmetric two-tensors $h,k$, we use the
inner products
\begin{equation}\label{eq:metric-convention}
	\langle\alpha,\gamma\rangle
	:=\frac12\sum_{i,j=1}^{2m}
	\alpha(e_i,e_j)\gamma(e_i,e_j),
	\qquad
	\langle h,k\rangle
	:=\sum_{i,j=1}^{2m}h(e_i,e_j)k(e_i,e_j),
\end{equation}
where $\{e_i\}_{i=1}^{2m}$ is any local orthonormal frame. For any
two-form $\alpha$, we set
\begin{equation*}
	\Lambda\alpha:=\langle\alpha,\omega\rangle.
\end{equation*}

We write $\nabla^{\ast}$ and $\delta$ for the formal $L^2$-adjoints of the Levi-Civita connection $\nabla$ and the exterior derivative $\dd$, respectively. Our conventions for the
Laplace--Beltrami, rough, and Hodge Laplacians are
\begin{equation*}
	\Delta_g:=\operatorname{tr}_g\nabla^2,
	\qquad
	\nabla^{\ast}\nabla:=-\operatorname{tr}_g\nabla^2,
	\qquad
	\Delta_H:=\dd\delta+\delta\dd.
\end{equation*}
Thus $\nabla^\ast\nabla$ and $\Delta_H$ are nonnegative in the $L^2$ sense, and $\Delta_Hu=-\Delta_gu$ for any smooth function $u$. For a vector field $X$ and a symmetric two-tensor $h$, we set
\begin{equation*}
	\divergence\! X
	:=\sum_{i=1}^{2m}g(\nabla_{e_i}X,e_i),
	\qquad
	(\divergence\! h)(Y)
	:=\sum_{i=1}^{2m}(\nabla_{e_i}h)(e_i,Y).
\end{equation*}

We adopt the curvature convention $R_{X,Y}:=\nabla_{[X,Y]}-[\nabla_X,\nabla_Y]$. The corresponding curvature operator $\mathcal R:\Lambda^2T^{\ast}M\longrightarrow\Lambda^2T^{\ast}M$ is characterized by
\begin{equation*}
	\left\langle
	\mathcal R(P^{\flat}\wedge Q^{\flat}),
	X^{\flat}\wedge Y^{\flat}
	\right\rangle
	=g(R_{P,Q}X,Y)
	=Rm(P,Q,X,Y),
\end{equation*}
for all tangent vectors $P,Q,X,Y$. The \emph{star-Ricci form} is then defined by
\begin{equation}\label{eq:rho*}
	\rho^{\ast}:=\mathcal R(\omega).
\end{equation}

\subsection{Almost-K\"ahler linear algebra}

We recall the basic almost-K\"ahler linear algebra used below; see \cite[\S~1.2]{Huybrechts:05} and \cite[\S~3]{Apostolov-Draghici:03} for more details.

An orthonormal frame $\{e_1,\ldots,e_{2m}\}$ satisfying $e_{2i}=Je_{2i-1}$, $1\leq i\leq m$, is called a {\em $J$-adapted frame}. We orient $M$ by declaring such frames positively oriented. If $\{e^1,\ldots,e^{2m}\}$ denotes the dual coframe, $\omega=e^1\wedge e^2+\cdots+e^{2m-1}\wedge e^{2m}$. Thus,
\begin{equation}\label{eq:volume-normalization}
	\dd\mu_g=\frac{\omega^m}{m!},
	\qquad
	|\omega|^2=m.
\end{equation}

Let $\ast$ denote the {\em Hodge star} determined by $g$ and the chosen orientation, characterized by
\begin{equation}\label{eq:hodge*}
	\alpha\wedge \ast \gamma
	=
	\langle\alpha,\gamma\rangle\dd\mu_g,
	\qquad
	\alpha,\gamma\in\Lambda^kT^*M.
\end{equation}
In particular,
\begin{equation}\label{eq:*omega}
	\ast\omega=\frac{\omega^{m-1}}{(m-1)!}.
\end{equation}
Since $\dd\omega=0$, it follows that $\delta\omega=-\ast\dd\ast\omega=0$. Thus $\omega$ is harmonic.

For two-forms, define the {\em $J$-invariant} and {\em $J$-anti-invariant} subspaces by
\begin{equation*}
	\Lambda_J^\pm:=\bigl\{ \alpha\in\Lambda^2: \alpha(JX,JY)=\pm\alpha(X,Y) \bigr\}.
\end{equation*}
A two-form $\alpha$ is called {\em primitive} if $\Lambda\alpha=0$; see, for example, \cite[Definition~1.2.29]{Huybrechts:05}. Accordingly, the primitive part of $\Lambda_J^+$ is
\begin{equation*}
	\Lambda_{J,0}^+
	:=
	\bigl\{
	\alpha\in\Lambda_J^+:
	\langle\alpha,\omega\rangle=0
	\bigr\}.
\end{equation*}
Then, we have the orthogonal decomposition of the two-forms
\begin{equation}\label{eq:J-form-decomposition}
	\Lambda^2
	=\mathbb R\omega\oplus\Lambda_{J,0}^+\oplus\Lambda_J^-;
\end{equation}
see, e.g., \cite[\S~3.1, equation~(5)]{Apostolov-Draghici:03}. In particular, every form in $\Lambda_J^-$ is primitive.

For a symmetric two-tensor $h$, let $h^\sharp$ be the self-adjoint endomorphism characterized by
\begin{equation*}
	g(h^\sharp X,Y)=h(X,Y).
\end{equation*}
Define the {\em $J$-invariant} and {\em $J$-anti-invariant} parts of $h$ by
\begin{equation*}
	\begin{aligned}
		h^{\jinv}(X,Y)
		&:=\frac12\bigl(h(X,Y)+h(JX,JY)\bigr),\\
		h^{\janti}(X,Y)
		&:=\frac12\bigl(h(X,Y)-h(JX,JY)\bigr).
	\end{aligned}
\end{equation*}
The decomposition $h=h^{\jinv}+h^{\janti}$ is orthogonal; the corresponding endomorphisms commute and anticommute with $J$, respectively, and $\operatorname{tr}_g h^{\jinv}=\operatorname{tr}_g h,\ \operatorname{tr}_g h^{\janti}=0$.

\begin{lemma}\label{lem:algebra}
	Let $E$ be the trace-free Ricci tensor and $E^{\jinv}$ its $J$-invariant part:
	\begin{equation}\label{eq:trace-free_Ricc}
		E:=Rc-\frac{R}{2m}g,
		\qquad
		E^{\jinv}:=Rc^{\jinv}-\frac{R}{2m}g.
	\end{equation}
	Define $\rho_0$ by $\rho_0(X,Y):=E^{\jinv}(JX,Y)$. Then we have $\rho_0\in\Lambda_{J,0}^+$ and
	\begin{equation*}
		|E|^2=2|\rho_0|^2+|Rc^{\janti}|^2.
	\end{equation*}
\end{lemma}

\begin{proof}
	First, by the symmetry and the $J$-invariance of $E^{\jinv}$, we have
	\begin{equation*}
		\rho_0(Y,X)
		=E^{\jinv}(JY,X)
		=-E^{\jinv}(JX,Y)
		=-\rho_0(X,Y),
	\end{equation*}
	and
	\begin{equation*}
		\rho_0(JX,JY)
		=E^{\jinv}(-X,JY)
		=E^{\jinv}(JX,Y)
		=\rho_0(X,Y).
	\end{equation*}
	Thus $\rho_0$ is a $J$-invariant two-form, called the {\em trace-free Ricci form}. 
	
	In a $J$-adapted frame, \eqref{eq:metric-convention} gives
	\begin{equation*}
		\langle\rho_0,\omega\rangle
		=\sum_{i=1}^m
		E^{\jinv}(e_{2i},e_{2i})
		=\frac12\operatorname{tr}_gE^{\jinv}
		=0,
	\end{equation*}
	where we used the $J$-invariance and trace-freeness of
	$E^{\jinv}$. Hence $\rho_0\in\Lambda_{J,0}^+$.
	
	Finally, by \eqref{eq:metric-convention} and the orthonormality of $\{Je_i\}_{i=1}^{2m}$,
	\begin{equation*}
		|\rho_0|^2
		=
		\frac12\sum_{i,j=1}^{2m}
		E^{\jinv}(Je_i,e_j)^2
		=
		\frac12|E^{\jinv}|^2.
	\end{equation*}
	Combining this with the orthogonal decomposition $E=E^{\jinv}+Rc^{\janti}$, we obtain
	\begin{equation*}
		|E|^2
		=
		|E^{\jinv}|^2+|Rc^{\janti}|^2
		=
		2|\rho_0|^2+|Rc^{\janti}|^2.
	\end{equation*}
	This finishes the proof of Lemma~\ref{lem:algebra}.
\end{proof}

We shall also need the following consequence of the classical Weil identities \cite{Weil:58}.

\begin{lemma}\label{lem:Weil-two-forms}
	If $a\in\mathbb R$, $\alpha\in\Lambda_{J,0}^+$, and $\gamma\in\Lambda_J^-$, then
	\begin{equation}\label{eq:Weil-quadratic-form}
		\frac{(a\omega+\alpha+\gamma)^2\wedge\omega^{m-2}}{(m-2)!}
		=\bigl(m(m-1)a^2-|\alpha|^2+|\gamma|^2\bigr)\dd\mu_g.
	\end{equation}
\end{lemma}

\begin{proof}
	Since $\alpha\in\Lambda_{J,0}^+$ and $\gamma\in\Lambda_J^-$ are primitive, the degree-two Weil identity (see \cite[equation~(1.16)]{Weil:58} or \cite[Proposition~1.2.31 \& Example 1.2.32]{Huybrechts:05}) yields
	\begin{equation}\label{eq:Weil-star}
		\ast \alpha=-\frac{\alpha\wedge\omega^{m-2}}{(m-2)!},
		\qquad
		\ast \gamma=\frac{\gamma\wedge\omega^{m-2}}{(m-2)!}.
	\end{equation}
	Consequently, by \eqref{eq:hodge*},
	\begin{equation*}
		\frac{\alpha^2\wedge\omega^{m-2}}{(m-2)!}
		=-|\alpha|^2\dd\mu_g,\qquad
		\frac{\gamma^2\wedge\omega^{m-2}}{(m-2)!}
		=|\gamma|^2\dd\mu_g.
	\end{equation*}
	Since $\alpha$ and $\gamma$ are primitive and orthogonal, \eqref{eq:hodge*}, \eqref{eq:*omega}, and \eqref{eq:Weil-star} imply
	\begin{equation*}
		\alpha\wedge\omega^{m-1}
		=\gamma\wedge\omega^{m-1}=0,
		\qquad
		\alpha\wedge\gamma\wedge\omega^{m-2}=0.
	\end{equation*}
	Finally, \eqref{eq:volume-normalization} yields
	\begin{equation*}
		\frac{\omega^m}{(m-2)!}
		=m(m-1)\dd\mu_g.
	\end{equation*}
	These identities yield the result upon expanding the left-hand side of \eqref{eq:Weil-quadratic-form}.
\end{proof}

\subsection{Some basic facts about gradient shrinking Ricci solitons}

We recall the following basic identities satisfied by gradient shrinking Ricci solitons.

\begin{lemma} [Hamilton \cite{Hamilton:95}] \label{lem:Hamilton}
	Let $(M^n,g,f)$ be an $n$-dimensional complete gradient shrinking Ricci soliton satisfying Eq. \eqref{eq:soliton}. Then
	\begin{itemize}
		\item[(i)] $ R + \Delta_g f = n\lambda, $

		\item[(ii)] $ \dd R = 2Rc(\nabla f,\mathord\cdot),$

		\item[(iii)] $ R + |\nabla f|^2 = 2\lambda f+C_0,$
	\end{itemize}
	where $C_0$ is some constant.
\end{lemma}

Throughout the paper, we normalize the soliton by rescaling $g$ and adding a constant to $f$ so that
\begin{equation}\label{eq:normalized-f}
	\lambda=\frac{1}{2},
	\qquad
	R+|\nabla f|^2=f.
\end{equation}
Nevertheless, we retain $\lambda$ in formulas for notational convenience.

Finally, we state several fundamental facts about complete gradient shrinking Ricci solitons that we shall need later. 

\begin{lemma} [Cao--Zhou \cite{Cao-Zhou:10}] \label{lem:Cao-Zhou}
	Let $(M^n,g,f)$ be an $n$-dimensional complete gradient shrinking Ricci soliton, and fix $p\in M$. Then there exist positive constants $c_1$, $c_2$ and $C$ such that, for any $x\in M$ and $r>0$,
	\begin{equation*}
		\begin{gathered}
			\frac{1}{4}(d(x,p)-c_1)_+^2 \leq f(x) \leq \frac{1}{4}(d(x,p)+c_2)^2, \\
			\mathrm{Vol}(B_p(r)) \leq Cr^n.
		\end{gathered}
	\end{equation*}
\end{lemma}

\begin{lemma} [Munteanu--\v{S}e\v{s}um \cite{Munteanu-Sesum:13}] \label{lem:Munteanu-Sesum}
	Let $(M^n,g,f)$ be an $n$-dimensional complete gradient shrinking Ricci soliton. Then, for any $a>0$, we have
	\begin{equation*}
		\int_M e^{-a f} |Rc|^2 \dd\mu_g< \infty. 
	\end{equation*}
\end{lemma}

\section{A coercive identity}\label{sec:coercive}

Throughout the remainder of the paper, for convenience, we set
\begin{equation*}
	V:=\nabla f,
	\qquad
	W:=JV,
	\qquad
	\beta:=\Lie_V\omega.
\end{equation*}
Since $\iota_V\omega=W^\flat$ and $\dd\omega=0$, Cartan's formula implies
\begin{equation}\label{eq:beta}
	\beta=\dd (\iota_V\omega)=\dd W^\flat.
\end{equation}
Thus $\beta$ is exact and hence closed. We first determine its components with respect to the decomposition \eqref{eq:J-form-decomposition}.

\begin{lemma}\label{lem:Lie-decomposition}
	Let $(M^{2m},g,J,\omega,f)$, $m\geq 2$, be an almost-K\"ahler gradient shrinking Ricci soliton. Then, for $V=\nabla f$ and $W=JV$, we have
	\begin{equation}\label{eq:beta-decomp}
		\beta:=\Lie_V\omega
		=\frac{\Delta_gf}{m}\omega-2\rho_0+\nabla_V\omega,
	\end{equation}
	where $\rho_0$ is the trace-free Ricci form. Moreover, 
	\begin{equation}\label{eq:na_V-omega}
		\nabla_V\omega\in\Lambda_J^-,
		\qquad
		(\nabla_V\omega)(V,W)=0,
	\end{equation}
	and
	\begin{equation}\label{eq:beta-VW}
		\beta(V,W)=\Hess f(V,V)+\Hess f(W,W).
	\end{equation}
\end{lemma}

\begin{proof}
	First, since $V=\nabla f$, we have $\nabla_XV=(\Hess f)^\sharp X$. Expanding the Lie derivative gives
	\begin{equation}\label{eq:Lie-omega}
		(\Lie_V\omega)(X,Y)
		=(\nabla_V\omega)(X,Y)
		+\omega\bigl((\Hess f)^\sharp X,Y\bigr)
		+\omega\bigl(X,(\Hess f)^\sharp Y\bigr).
	\end{equation}
	The soliton equation \eqref{eq:soliton} and \eqref{eq:trace-free_Ricc} imply
	\begin{equation*}
		\Hess f
		=\left(\lambda-\frac{R}{2m}\right)g
		-E^{\jinv}-Rc^{\janti}.
	\end{equation*}
	Since $(E^{\jinv})^\sharp$ commutes with $J$ and $(Rc^{\janti})^\sharp$ anticommutes with $J$, we have
	\begin{equation*}
		\begin{aligned}
			\omega\bigl((E^{\jinv})^\sharp X,Y\bigr)
			+\omega\bigl(X,(E^{\jinv})^\sharp Y\bigr)
			&=2\rho_0(X,Y),\\
			\omega\bigl((Rc^{\janti})^\sharp X,Y\bigr)
			+\omega\bigl(X,(Rc^{\janti})^\sharp Y\bigr)
			&=0.
		\end{aligned}
	\end{equation*}
	Substituting these three identities into \eqref{eq:Lie-omega} and using Lemma~\ref{lem:Hamilton}(i) proves \eqref{eq:beta-decomp}.
	
	Next, since $\omega(X,Y)=g(JX,Y)$ and $\nabla g=0$, we have
	\begin{equation}\label{eq:nablaV-omega}
		(\nabla_V\omega)(X,Y)
		=g\bigl((\nabla_VJ)X,Y\bigr).
	\end{equation}
	Differentiating $J^2=-\operatorname{Id}$ along $V$ yields
	\begin{equation}\label{eq:nablaVJ-anticommute}
		(\nabla_VJ)J=-J(\nabla_VJ).
	\end{equation}
	It follows from \eqref{eq:nablaV-omega} and \eqref{eq:nablaVJ-anticommute} that
	\begin{equation*}
		(\nabla_V\omega)(JX,JY)
		=-(\nabla_V\omega)(X,Y),
	\end{equation*}
	and hence $\nabla_V\omega\in\Lambda_J^-$. In particular, since $W=JV$,
	\begin{equation*}
		(\nabla_V\omega)(V,W)=0.
	\end{equation*}
	
	Finally, evaluating \eqref{eq:Lie-omega} on $(V,W)$ yields
	\begin{equation*}
		\beta(V,W)
		=\omega\bigl((\Hess f)^\sharp V,W\bigr)
		+\omega\bigl(V,(\Hess f)^\sharp W\bigr)
		=\Hess f(V,V)+\Hess f(W,W).
	\end{equation*}
	This proves \eqref{eq:beta-VW}, and thereby concludes the proof of Lemma~\ref{lem:Lie-decomposition}.
\end{proof}

The following two identities form the analytic core of the argument.

\begin{lemma}\label{lem:two-identities}
	Let $(M^{2m},g,J,\omega,f)$, $m\geq2$, be a complete almost-K\"ahler
	gradient shrinking Ricci soliton. Then, for $V=\nabla f$ and $W=JV$, we have
	\begin{align}
		0&=\int_M e^{-f/2}\left(
		2|Rc^{\janti}|^2+|\nabla_V\omega|^2
		+Rc^{\janti}(V,V)\right)\dd\mu_g,
		\label{eq:exact-half}\\
		0&=\int_M e^{-f/2}\left(
		2|Rc^{\janti}|^2+\frac12|\Lie_Wg|^2
		-|\nabla_V\omega|^2
		-2Rc^{\janti}(V,V)\right)\dd\mu_g.
		\label{eq:Yano-half}
	\end{align}
	Moreover, each term in the two integrands is integrable with respect to $e^{-f/2}\dd\mu_g$.
\end{lemma}

We first prove the two identities in the compact case and then extend them to the complete noncompact case using cutoffs adapted to the potential function $f$. The proof is divided into three parts.

\begin{proof}[\bf Part I: proof of \eqref{eq:exact-half} in the compact case.]
	On the one hand, by Lemma~\ref{lem:Lie-decomposition},
	\begin{equation*}
		\beta
		=\frac{\Delta_gf}{m}\omega-2\rho_0+\nabla_V\omega.
	\end{equation*}
	Since $\rho_0\in\Lambda_{J,0}^+$ and $\nabla_V\omega\in\Lambda_J^-$, Lemma~\ref{lem:Weil-two-forms} yields
	\begin{equation}\label{eq:beta-Weil}
		\frac{\beta^2\wedge\omega^{m-2}}{(m-2)!}
		=
		\left(
		\frac{m-1}{m}(\Delta_gf)^2
		-4|\rho_0|^2+|\nabla_V\omega|^2
		\right)\dd\mu_g.
	\end{equation}
	
	On the other hand, since 
	\begin{equation*}
		\dd f(JX)=-W^\flat(X),\qquad W^\flat(JX)=\dd f(X),
	\end{equation*}
	the two-form $\dd f\wedge W^\flat$ is $J$-invariant. Moreover,
	\begin{equation*}
		\Lambda(\dd f\wedge W^\flat)=|V|^2,
		\qquad
		\Lambda\beta=\Delta_gf,
		\qquad
		\langle\dd f\wedge W^\flat,\beta\rangle=\beta(V,W).
	\end{equation*}
	The polarized form of Lemma~\ref{lem:Weil-two-forms} (see also \cite[Lemma~4.7]{Szekelyhidi:14}) therefore gives
	\begin{equation*}
		\frac{\dd f\wedge W^\flat\wedge\beta
			\wedge\omega^{m-2}}{(m-2)!}
		=
		\left(
		|V|^2\Delta_gf-\beta(V,W)
		\right)\dd\mu_g.
	\end{equation*}
	Using $\beta=\dd W^\flat$, $\dd \beta=0$, and $\dd \omega=0$, we have
	\begin{equation*}
		\dd\left(W^\flat\wedge\beta\wedge\omega^{m-2}\right)=\beta^2\wedge\omega^{m-2}.
	\end{equation*}
	Thus, Stokes' theorem yields
	\begin{equation}\label{eq:half-Weil-Stokes}
		\begin{aligned}
			\int_Me^{-f/2}
			\frac{\beta^2\wedge\omega^{m-2}}{(m-2)!}
			&=-\int_M\dd(e^{-f/2})\wedge
			\frac{W^\flat\wedge\beta\wedge\omega^{m-2}}{(m-2)!}\\
			&=\frac12\int_Me^{-f/2}
			\frac{\dd f\wedge W^\flat\wedge\beta
				\wedge\omega^{m-2}}{(m-2)!}\\
			&=\frac12\int_Me^{-f/2}
			\left(
			|V|^2\Delta_gf-\beta(V,W)
			\right)\dd\mu_g.
		\end{aligned}
	\end{equation}
	
	Next, taking the trace-free part of the soliton equation \eqref{eq:soliton} and using \eqref{eq:trace-free_Ricc}, the contracted Bianchi identity, and Lemma~\ref{lem:Hamilton}(i), we obtain
	\begin{equation}\label{eq:E&divE}
		E=-(\Hess f)_0=-\left( \Hess f-\frac{\Delta_gf}{2m}g\right),
		\qquad
		\divergence\! E
		=-\frac{m-1}{2m}\,\dd(\Delta_gf).
	\end{equation}
	Since $E$ is trace-free, \eqref{eq:E&divE} and weighted integration by parts yield
	\begin{equation}\label{eq:half-trace-free-energy}
		\begin{split}
			\int_Me^{-f/2}|E|^2\dd\mu_g
			&=-\int_Me^{-f/2}\langle E,\Hess f\rangle\dd\mu_g\\
			&=\int_Me^{-f/2}(\divergence\! E)(V)\dd\mu_g
			-\frac12\int_Me^{-f/2}E(V,V)\dd\mu_g\\
			&=\frac{m-1}{2m}\int_Me^{-f/2}
			\left(
			(\Delta_gf)^2
			-\frac12(\Delta_gf)|V|^2
			\right)\dd\mu_g\\
			&\quad-\frac12\int_Me^{-f/2}E(V,V)\dd\mu_g,
		\end{split}
	\end{equation}
	where the last equality follows from $ \divergence(e^{-f/2}V)=e^{-f/2}\left(\Delta_gf-\frac12|V|^2\right)$.
	
	After substituting \eqref{eq:beta-Weil} into the left-hand side of \eqref{eq:half-Weil-Stokes}, we multiply \eqref{eq:half-trace-free-energy} by $2$ and use it to eliminate the term involving $(\Delta_gf)^2$. Then Lemma~\ref{lem:algebra} gives
	\begin{equation*}
		0=\int_Me^{-f/2}\Bigg(
		2|Rc^{\janti}|^2+|\nabla_V\omega|^2
		+\frac12\Big[
		-\frac1m(\Delta_gf)|V|^2
		+2E(V,V)+\beta(V,W)
		\Big]\Bigg)\dd\mu_g.
	\end{equation*}
	Finally, by Lemma~\ref{lem:Hamilton}(i), \eqref{eq:trace-free_Ricc}, \eqref{eq:beta-VW}, the soliton equation~\eqref{eq:soliton}, and $|W|=|V|$,
	\begin{equation}\label{eq:Rc-anti}
		\begin{split}
			&-\frac1m(\Delta_gf)|V|^2
			+2E(V,V)+\beta(V,W)\\
			&\qquad
			=-2\lambda|V|^2+2Rc(V,V)
			+\Hess f(V,V)+\Hess f(W,W)\\
			&\qquad
			=-\Hess f(V,V)+\Hess f(W,W)\\
			&\qquad
			=Rc(V,V)-Rc(JV,JV)=2Rc^{\janti}(V,V).
		\end{split}
	\end{equation}
	Substituting this identity into the preceding integral proves \eqref{eq:exact-half}.
\end{proof}

\begin{proof}[\bf Part II: proof of \eqref{eq:Yano-half} in the compact case.]
	First, as $\iota_W\omega=-\dd f$ and $\dd\omega=0$, Cartan's formula gives $\Lie_W\omega=0$. Hence,
	\begin{equation*}
		\left( \divergence\! W\right) \dd\mu_g
		=\Lie_W\left(\frac{\omega^m}{m!}\right)
		=\frac{\Lie_W\omega\wedge\omega^{m-1}}{(m-1)!}=0,
	\end{equation*}
	which implies $\divergence\! W=0$. Moreover, $Wf=\dd f(W)=g(V,W)=0$.
	Consequently,
	\begin{equation}\label{eq:weighted-delta}
		\delta W^\flat+\frac12\iota_VW^\flat=0.
	\end{equation}
	It is also easy to see that
	\begin{equation}\label{eq:weighted-Rc}
		Rc-\Hess\bigl(\log e^{-f/2}\bigr)=Rc+\frac12\Hess f.
	\end{equation}
	Now, applying Lott's weighted Bochner identity \cite[equation~(2.12)]{Lott:03} to $W^\flat$ and using \eqref{eq:weighted-Rc}, we have
	\begin{align*}
		\int_Me^{-f/2}|\nabla W^{\flat}|^2\dd\mu_g
		&=\int_Me^{-f/2}\left(
		|\dd W^\flat|^2
		+\left|\delta W^\flat+\frac12\iota_VW^\flat\right|^2
		\right)\dd\mu_g\\
		&\quad-\int_Me^{-f/2}
		\left(Rc+\frac12\Hess f\right)(W,W)\dd\mu_g\\
		&=\int_Me^{-f/2}|\beta|^2\dd\mu_g
		-\int_Me^{-f/2}
		\left(Rc+\frac12\Hess f\right)(W,W)\dd\mu_g,
	\end{align*}
	where we have used \eqref{eq:weighted-delta} and \eqref{eq:beta} in the last equality. 
	
	The symmetric and skew-symmetric parts of $\nabla W^\flat$ are $\frac12\Lie_Wg$ and $\frac12\dd W^\flat$, respectively. Their orthogonality, $\dd W^\flat=\beta$, and the norm conventions in \eqref{eq:metric-convention} yield
	\begin{equation}\label{eq:norm_na-W}
		|\beta|^2+\frac12|\Lie_Wg|^2=2|\nabla W^{\flat}|^2.
	\end{equation}
	Combining the weighted integral of \eqref{eq:norm_na-W} with the preceding identity, we obtain
	\begin{equation}\label{eq:half-Yano-raw}
		\begin{split}
			\frac12\int_Me^{-f/2}|\Lie_Wg|^2\dd\mu_g
			&=\int_Me^{-f/2}|\beta|^2\dd\mu_g\\
			&\quad-2\int_Me^{-f/2}
			\left(Rc+\frac12\Hess f\right)(W,W)\dd\mu_g.
		\end{split}
	\end{equation}
	
	Next, we compute the weighted energy of $\beta$. By
	\eqref{eq:beta-decomp}, the orthogonality of the three components of $\beta$, and Lemma~\ref{lem:algebra}, we have
	\begin{equation*}
		\int_Me^{-f/2}|\beta|^2\dd\mu_g
		=\int_Me^{-f/2}\left(
		\frac1m(\Delta_gf)^2+2|E|^2
		-2|Rc^{\janti}|^2+|\nabla_V\omega|^2
		\right)\dd\mu_g.
	\end{equation*}
	It remains to evaluate the first two terms on the right-hand side of the above identity. Using \eqref{eq:half-trace-free-energy}, \eqref{eq:trace-free_Ricc}, and Lemma~\ref{lem:Hamilton}(i), we obtain
	\begin{align*}
		&\int_Me^{-f/2}\left(
		\frac1m(\Delta_gf)^2+2|E|^2\right)\dd\mu_g\\
		&\qquad=
		\int_Me^{-f/2}\left[
		(\Delta_gf)\left(\Delta_gf-\frac12|V|^2\right)
		-Rc(V,V)+\lambda|V|^2\right]\dd\mu_g.
	\end{align*}
	Moreover, by $ \divergence(e^{-f/2}V)=e^{-f/2}\left(\Delta_gf-\frac12|V|^2\right)$, Lemma~\ref{lem:Hamilton}(i)\&(ii), and the soliton equation \eqref{eq:soliton}, weighted integration by parts gives
	\begin{align*}
		\int_Me^{-f/2}(\Delta_gf)
		\left(\Delta_gf-\frac12|V|^2\right)\dd\mu_g
		&=-\int_Me^{-f/2}
		\langle\nabla(\Delta_gf),V\rangle\dd\mu_g\\
		&=2\int_Me^{-f/2}Rc(V,V)\dd\mu_g\\
		&=2\int_Me^{-f/2}
		\left(\lambda|V|^2-\Hess f(V,V)\right)\dd\mu_g.
	\end{align*}
	Consequently, combining the preceding three identities and using the soliton equation \eqref{eq:soliton}, we obtain
	\begin{equation}\label{eq:half-beta-energy}
		\begin{split}
			&\int_Me^{-f/2}|\beta|^2\dd\mu_g\\
			&\quad=
			\int_Me^{-f/2}\left(
			2\lambda|V|^2-\Hess f(V,V)
			-2|Rc^{\janti}|^2+|\nabla_V\omega|^2
			\right)\dd\mu_g.
		\end{split}
	\end{equation}
	
	Finally, the soliton equation \eqref{eq:soliton} and $|W|=|V|$ give
	\begin{equation*}
		Rc+\frac12\Hess f
		=\lambda g-\frac12\Hess f,
		\qquad
		-\Hess f(V,V)+\Hess f(W,W)
		=2Rc^{\janti}(V,V).
	\end{equation*}
	Substituting \eqref{eq:half-beta-energy} and these two identities into \eqref{eq:half-Yano-raw}, with $|W|=|V|$, yields
	\begin{equation*}
		\frac12\int_Me^{-f/2}|\Lie_Wg|^2\dd\mu_g
		=
		\int_Me^{-f/2}\left(
		-2|Rc^{\janti}|^2+|\nabla_V\omega|^2
		+2Rc^{\janti}(V,V)\right)\dd\mu_g.
	\end{equation*}
	Rearranging proves \eqref{eq:Yano-half}.
\end{proof}

\begin{proof}[\bf Part III: proof of the complete noncompact case.]
	Assume now that $M$ is complete and noncompact. By Chen's result \cite{Chen:09}, $R\geq0$. Thus, \eqref{eq:normalized-f} gives
	\begin{equation*}
		f=R+|\nabla f|^2\geq0.
	\end{equation*} 
	For each integer $\ell\geq2$, define the cutoff function
	\begin{equation*}
		\psi_{\ell}:=\vartheta\left(\tfrac{f}{\ell}\right),
	\end{equation*}
	where $\vartheta\in C^\infty([0,\infty),[0,1])$ is nonincreasing, equals $1$ on $[0,1]$, vanishes on $[2,\infty)$, and satisfies
	\begin{equation*}
		|\vartheta'|^2\leq C\vartheta,\qquad|\vartheta''|\leq C
	\end{equation*}
	for some constant $C>0$ independent of $\ell$. By Lemma~\ref{lem:Cao-Zhou}, $f$ is proper, so $\psi_{\ell}$ has compact support in $\{f\leq2\ell\}$. Moreover, $\psi_\ell\uparrow1$ pointwise as $\ell\to\infty$, and
	\begin{equation*}
		\nabla\psi_{\ell}
		=\tfrac{1}{\ell}\vartheta'\left(\tfrac{f}{\ell}\right)V,
		\qquad
		W\psi_{\ell}
		=\tfrac{1}{\ell}\vartheta'\left(\tfrac{f}{\ell}\right)Wf
		=0.
	\end{equation*}
	Thus $\nabla\psi_{\ell}$ is parallel to $V$ and supported in $\{\ell \leq f \leq 2\ell\}$. On this transition region, $|V|^2\leq f$ and the properties of $\vartheta$ yield
	\begin{equation}\label{eq:half-cutoff-bounds}
		|V|^2\leq C\ell,
		\qquad
		\tfrac{|\nabla\psi_{\ell}|^2}{\psi_{\ell}}\leq\tfrac{C}{\ell},
		\qquad
		\psi_{\ell}|\nabla\psi_{\ell}||V|\leq C,
	\end{equation}
	where the second inequality is understood on $\{\psi_{\ell}>0\}$. In the following, we write $\psi=\psi_\ell$ for simplicity.
	
	\medskip\noindent\emph{Proof of \eqref{eq:exact-half} in the complete noncompact case.}
	We repeat the compact-case computation with the cutoff $\psi^2$ inserted. Since $W=JV$, we have
	\begin{equation*}
		\Lambda(\dd \psi\wedge W^\flat)
		=\langle\nabla\psi,V\rangle,
		\qquad
		\langle\dd\psi\wedge W^\flat,\beta\rangle
		=\beta(\nabla\psi,W).
	\end{equation*}
	Moreover, $\dd\psi\wedge W^\flat$ is $J$-invariant because $\nabla\psi$ is parallel to $V$. The polarized Weil identity and Stokes' theorem then yield
	\begin{equation*}
		\begin{split}
			&\int_M\psi^2e^{-f/2}
			\left[
			\frac{m-1}{m}(\Delta_gf)^2
			-2|E|^2+2|Rc^{\janti}|^2
			+|\nabla_V\omega|^2
			\right]\dd\mu_g\\
			&\quad=
			\frac12\int_M\psi^2e^{-f/2}
			\left[
			|V|^2\Delta_gf-\beta(V,W)
			\right]\dd\mu_g\\
			&\qquad
			+2\int_M\psi e^{-f/2}
			\left[
			\beta(\nabla\psi,W)
			-(\Delta_gf)\langle\nabla\psi,V\rangle
			\right]\dd\mu_g.
		\end{split}
	\end{equation*}
	Similarly, repeating the argument for \eqref{eq:half-trace-free-energy} with the same cutoff gives
	\begin{equation*}
		\begin{split}
			\int_M\psi^2e^{-f/2}|E|^2\dd\mu_g
			&=\frac{m-1}{2m}
			\int_M\psi^2e^{-f/2}
			\left[
			(\Delta_gf)^2-\frac12(\Delta_gf)|V|^2
			\right]\dd\mu_g\\
			&\quad-\frac12\int_M\psi^2e^{-f/2}E(V,V)\dd\mu_g
			+2\int_M\psi e^{-f/2}
			E(\nabla\psi,V)\dd\mu_g\\
			&\quad+\frac{m-1}{m}
			\int_M\psi e^{-f/2}(\Delta_gf)
			\langle\nabla\psi,V\rangle\dd\mu_g.
		\end{split}
	\end{equation*}
	Combining the above two identities with \eqref{eq:Rc-anti}, we obtain
	\begin{equation}\label{eq:localized-half-coercive}
		\begin{aligned}
			&\int_M\psi^2e^{-f/2}\left(
			2|Rc^{\janti}|^2+|\nabla_V\omega|^2
			+Rc^{\janti}(V,V)\right)\dd\mu_g\\
			&\quad=
			\int_M\psi e^{-f/2}\left(
			4E(\nabla\psi,V)
			+2\beta(\nabla\psi,W)
			-\frac2m(\Delta_gf)
			\langle\nabla\psi,V\rangle
			\right)\dd\mu_g.
		\end{aligned}
	\end{equation}
	
	Now, since $\nabla\psi$ is parallel to $V$, \eqref{eq:na_V-omega} gives $(\nabla_V\omega)(\nabla\psi,W)=0$. It follows from \eqref{eq:beta-decomp} and Lemma~\ref{lem:algebra} that
	\begin{equation*}
		|\beta(\nabla\psi,W)|
		\leq C\bigl(|\Delta_gf|+|E|\bigr)
		|\nabla\psi||V|.
	\end{equation*}
	Therefore, using \eqref{eq:half-cutoff-bounds}, the absolute value of the right-hand side of \eqref{eq:localized-half-coercive} is bounded by
	\begin{equation*}
		C\int_{\{\ell\leq f\leq2\ell\}}
		e^{-f/2}\bigl(|\Delta_gf|+|E|\bigr)\dd\mu_g.
	\end{equation*}
	By Lemma~\ref{lem:Hamilton}(i) and $|R|\leq \sqrt{2m}|Rc|$, we have $|\Delta_gf|+|E|\leq C(1+|Rc|)$. 
	Then, Lemmas~\ref{lem:Cao-Zhou} and~\ref{lem:Munteanu-Sesum} show that $|\Delta_gf|+|E|$ belongs to
	$L^1(e^{-f/2}\dd\mu_g)$. The preceding integral is therefore a weighted $L^1$ tail and tends to zero as $\ell\to\infty$.
	
	Moreover, $|V|^2\leq f$, Lemmas~\ref{lem:Cao-Zhou} and~\ref{lem:Munteanu-Sesum}, and Cauchy--Schwarz imply that $Rc^{\janti}(V,V)\in L^1(e^{-f/2}\dd\mu_g)$. Moving this term to the right-hand side of \eqref{eq:localized-half-coercive} and discarding the nonnegative term $2|Rc^{\janti}|^2$ gives
	\begin{equation*}
		\int_M\psi^2e^{-f/2}|\nabla_V\omega|^2\dd\mu_g
		\leq
		C+\int_Me^{-f/2}|Rc^{\janti}(V,V)|\dd\mu_g.
	\end{equation*}
	Fatou's lemma and $\psi_{\ell}\to1$ as $\ell \rightarrow \infty$ therefore imply $\nabla_V\omega\in L^2(e^{-f/2}\dd\mu_g)$. Hence, all terms on the left-hand side of \eqref{eq:localized-half-coercive} are now integrable.
	
	Finally, let $\ell\to\infty$ in \eqref{eq:localized-half-coercive}.
	Dominated convergence on the left-hand side and the vanishing
	of the right-hand side yield \eqref{eq:exact-half}. The decomposition \eqref{eq:beta-decomp} also gives
	\begin{equation}\label{eq:beta-weighted-L2}
		\beta\in L^2(e^{-f/2}\dd\mu_g).
	\end{equation}

	\medskip\noindent\emph{Proof of \eqref{eq:Yano-half} in the complete noncompact case.}
	We shall first recover \eqref{eq:half-beta-energy}. By Lemma~\ref{lem:Hamilton}(i)\&(ii), $\dd(\Delta_gf)=-2Rc(V,\mathord\cdot)$. Then, integration by parts with cutoff $\psi^2$ gives
	\begin{align*}
		2\int_M\psi^2e^{-f/2}Rc(V,V)\dd\mu_g
		&=\int_M\psi^2e^{-f/2}(\Delta_gf)
		\left(\Delta_gf-\frac12|V|^2\right)\dd\mu_g\\
		&\quad+2\int_M\psi e^{-f/2}(\Delta_gf)
		\langle\nabla\psi,V\rangle\dd\mu_g.
	\end{align*}
	Recall that repeating the argument for \eqref{eq:half-trace-free-energy}
	with cutoff $\psi^2$ introduces the terms
	\begin{equation*}
		\int_M\psi e^{-f/2}\left[
		2E(\nabla\psi,V)
		+\frac{m-1}{m}(\Delta_gf)
		\langle\nabla\psi,V\rangle
		\right]\dd\mu_g.
	\end{equation*}
	By \eqref{eq:half-cutoff-bounds}, all terms involving $\nabla\psi$
	in both computations are bounded in absolute value by
	\begin{equation*}
		C\int_{\{\ell\leq f\leq2\ell\}}
		e^{-f/2}\bigl(|\Delta_gf|+|E|\bigr)\dd\mu_g,
	\end{equation*}
	which tends to zero as $\ell\to\infty$. Since the remaining integrands are integrable with respect to $e^{-f/2}\dd\mu_g$, we may pass to the limit by dominated convergence. Combining the resulting identities as in the compact case yields \eqref{eq:half-beta-energy}.
	
	Next, since $W\psi=0$, the product rule for the codifferential and \eqref{eq:weighted-delta} imply
	\begin{equation*}
		\delta(\psi W^\flat)
		+\frac12\iota_V(\psi W^\flat)
		=
		\psi\left(
		\delta W^\flat+\frac12\iota_VW^\flat
		\right)-W\psi
		=0.
	\end{equation*}
	Applying Lott's weighted Bochner identity \cite[equation~(2.12)]{Lott:03} to $\psi W^\flat$ and using the same norm decomposition as in
	\eqref{eq:norm_na-W}, we obtain
	\begin{equation*}
		0=\int_M e^{-f/2}\bigg(
		\frac12|\Lie_{\psi W}g|^2
		-\left|\dd(\psi W^\flat)\right|^2
		+2\left(Rc+\frac12\Hess f\right)
		(\psi W,\psi W)
		\bigg)\,\dd\mu_g.
	\end{equation*}
	A direct expansion, using $\dd W^\flat=\beta$ and $W\psi=0$, gives
	\begin{equation*}
		\begin{split}
			&\frac12|\Lie_{\psi W}g|^2
			-\left|\dd(\psi W^\flat)\right|^2\\
			&\quad=
			\psi^2\left(
			\frac12|\Lie_Wg|^2-|\beta|^2
			\right)
			+2\psi\left(
			(\Lie_Wg)(\nabla\psi,W)
			-\beta(\nabla\psi,W)
			\right).
		\end{split}
	\end{equation*}
	Substituting this identity into the preceding integral identity, we have
	\begin{equation}\label{eq:localized-half-Yano}
		\begin{split}
			&\int_M\psi^2e^{-f/2}\left(
			\frac12|\Lie_Wg|^2-|\beta|^2
			+2\left(Rc+\frac12\Hess f\right)(W,W)
			\right)\dd\mu_g\\
			&\quad=2\int_M\psi e^{-f/2}\left(
			\beta(\nabla\psi,W)
			-(\Lie_Wg)(\nabla\psi,W)
			\right)\dd\mu_g.
		\end{split}
	\end{equation}
	
	To pass to the limit in the preceding identity, we shall first prove that $\Lie_Wg\in L^2(e^{-f/2}\dd\mu_g)$. Using $|W|=|V|$ and Young's inequality, we obtain
	\begin{equation*}
		2\psi|\nabla\psi||V|
		\bigl(|\beta|+|\Lie_Wg|\bigr)
		\leq
		\psi^2|\beta|^2
		+\frac14\psi^2|\Lie_Wg|^2
		+C|\nabla\psi|^2|V|^2.
	\end{equation*}
	Applying this estimate to the right-hand side of \eqref{eq:localized-half-Yano} and absorbing the term $\frac{1}{4}\psi^2|\Lie_Wg|^2$ into the left-hand side gives
	\begin{equation}\label{eq:L_Wg}
		\begin{split}
			\frac14\int_M\psi^2e^{-f/2}|\Lie_Wg|^2\dd\mu_g
			&\leq2\int_M\psi^2e^{-f/2}|\beta|^2\dd\mu_g\\
			&\quad+2\int_M\psi^2e^{-f/2}
			\left|\left(Rc+\frac12\Hess f\right)(W,W)\right|\dd\mu_g\\
			&\quad+C\int_Me^{-f/2}|\nabla\psi|^2|V|^2\dd\mu_g.
		\end{split}
	\end{equation}
	The soliton equation \eqref{eq:soliton} gives
	\begin{equation*}
		\left|\left(Rc+\frac12\Hess f\right)(W,W)\right| =\left|\frac12(Rc+\lambda g)(W,W)\right|
		\leq C(1+|Rc|)|V|^2,
	\end{equation*}
	which belongs to $L^1(e^{-f/2}\dd\mu_g)$ by Lemmas~\ref{lem:Cao-Zhou} and \ref{lem:Munteanu-Sesum}. Moreover, \eqref{eq:half-cutoff-bounds} yields
	\begin{equation*}
		\int_Me^{-f/2}|\nabla\psi|^2|V|^2\dd\mu_g
		\leq
		C\int_{\{\ell\leq f\leq2\ell\}}e^{-f/2}\dd\mu_g.
	\end{equation*}
	By \eqref{eq:beta-weighted-L2} and the preceding two estimates, the right-hand side of \eqref{eq:L_Wg} is bounded independently of $\ell$. Fatou's lemma therefore yields $\Lie_Wg\in L^2(e^{-f/2}\dd\mu_g)$.
	
	Now, for $T=\beta$ or $T=\Lie_Wg$, Cauchy--Schwarz and \eqref{eq:half-cutoff-bounds} imply
	\begin{equation*}
		\left|\int_M\psi e^{-f/2}
		T(\nabla\psi,W)\dd\mu_g\right|
		\leq
		C\left(
		\int_{\{\ell\leq f\leq2\ell\}}
		e^{-f/2}|T|^2\dd\mu_g
		\right)^{1/2}
		\longrightarrow0.
	\end{equation*}
	Hence the right-hand side of \eqref{eq:localized-half-Yano} vanishes as $\ell\to\infty$, while dominated convergence on the left gives \eqref{eq:half-Yano-raw}. Combining this with \eqref{eq:half-beta-energy} and proceeding as in the compact case proves \eqref{eq:Yano-half}, together with the asserted integrability properties.
	
	This completes the proof of {\bf Part~III} and thereby the proof of Lemma~\ref{lem:two-identities}.
\end{proof}

\begin{proposition}\label{prop:first-rigidity}
	Let $(M^{2m},g,J,\omega,f)$, $m\geq2$, be a complete almost-K\"ahler
	gradient shrinking Ricci soliton. Then, for $V=\nabla f$ and $W=JV$, we have
	\begin{equation}\label{eq:coercive}
		\int_Me^{-f/2}\left(
		6|Rc^{\janti}|^2+|\nabla_V\omega|^2
		+\frac12|\Lie_Wg|^2
		\right)\dd\mu_g=0.
	\end{equation}
	Consequently, for any tangent vectors $X$ and $Y$,
	\begin{equation*}
		Rc(JX,JY)=Rc(X,Y),
		\qquad
		\nabla_V\omega=0,
		\qquad
		\Lie_Wg=0.
	\end{equation*}
\end{proposition}

\begin{proof}
	Adding twice \eqref{eq:exact-half} to \eqref{eq:Yano-half} yields \eqref{eq:coercive}, in which every term in the integrand is nonnegative; thus each vanishes identically. This proves Proposition~\ref{prop:first-rigidity}.
\end{proof}

\smallskip
\section{Torsion transversality}\label{sec:torsion}
In this section, we shall use Proposition~\ref{prop:first-rigidity} to show that the intrinsic torsion $\nabla\omega$ is transverse to $\operatorname{span}\{\nabla f, J\nabla f\}$. We then derive several additional consequences of the same proposition.

\begin{proposition}
	\label{prop:horizontal-torsion}
	Let $(M^{2m},g,J,\omega,f)$, $m\geq2$, be a complete almost-K\"ahler gradient shrinking Ricci soliton. Then, for $V=\nabla f$ and $W=JV$, we have
	\begin{equation*}
		(\nabla_X\omega)(Y,Z)=0
	\end{equation*}
	whenever any one of $X,Y,Z$ lies in
	$\operatorname{span}\{V,W\}$.
\end{proposition}

\begin{proof}
	Since $\omega(Y,Z)=g(JY,Z)$ and $\nabla g=0$, we have
	\begin{equation}\label{eq:na-omega&na-J}
		(\nabla_X\omega)(Y,Z)
		=g\bigl((\nabla_XJ)Y,Z\bigr).
	\end{equation}
	By this identity and the skew-symmetry of $\nabla_X\omega$ in $Y$ and $Z$, it suffices to prove
	\begin{equation}\label{eq:torsion-reduction}
		(\nabla_XJ)V=(\nabla_XJ)W=0
		\quad\text{for all }X,
		\qquad
		\nabla_VJ=\nabla_WJ=0.
	\end{equation}
	
	Now, we establish the first two identities. Since $W=JV$ and $\nabla_XV=(\Hess f)^\sharp X$,
	\begin{equation}\label{eq:na_XW}
		\nabla_XW
		=(\nabla_XJ)V+J(\Hess f)^\sharp X.
	\end{equation}
	The soliton equation \eqref{eq:soliton} and the $J$-invariance of $Rc$ imply that $(\Hess f)^\sharp$ commutes with $J$, so $J(\Hess f)^\sharp$ is skew-adjoint. Substituting \eqref{eq:na_XW} into the Killing equation $\Lie_Wg=0$ from Proposition~\ref{prop:first-rigidity} and using this skew-adjointness, we obtain
	\begin{equation}\label{eq:torsion-skew}
		g\bigl((\nabla_XJ)V,Y\bigr)
		+g\bigl((\nabla_YJ)V,X\bigr)=0.
	\end{equation}
	
	On the other hand, using $\dd\omega=0$ and
	$\nabla_V\omega=0$, we obtain
	\begin{equation}\label{eq:torsion-symmetric}
		\begin{aligned}
			0
			&=(\dd\omega)(V,X,Y)\\
			&=(\nabla_V\omega)(X,Y)
			+(\nabla_X\omega)(Y,V)
			+(\nabla_Y\omega)(V,X)\\
			&=-g\bigl((\nabla_XJ)V,Y\bigr)
			+g\bigl((\nabla_YJ)V,X\bigr),
		\end{aligned}
	\end{equation}
	where we have used \eqref{eq:na-omega&na-J} in the last equality. Comparing \eqref{eq:torsion-skew} and \eqref{eq:torsion-symmetric} yields
	\begin{equation*}
		g\bigl((\nabla_XJ)V,Y\bigr)=0
	\end{equation*}
	for all $X,Y$. Therefore,
	\begin{equation*}
		(\nabla_XJ)V=0.
	\end{equation*}
	Differentiating $J^2=-\operatorname{Id}$ and using $W=JV$, we then obtain
	\begin{equation*}
		(\nabla_XJ)W
		=(\nabla_XJ)(JV)
		=-J(\nabla_XJ)V
		=0.
	\end{equation*}
	
	It remains to prove the last two identities in
	\eqref{eq:torsion-reduction}. Proposition~\ref{prop:first-rigidity}
	gives $\nabla_V\omega=0$, equivalently $\nabla_VJ=0$ by \eqref{eq:na-omega&na-J}. Moreover, for any vector field $X$, every almost-K\"ahler structure satisfies
	\begin{equation}\label{eq:quasi-Kahler-identity}
		\nabla_{JX}J=-J\nabla_XJ;
	\end{equation}
	see \cite[equation~(7)]{Apostolov-Draghici:03}. Taking $X=V$ gives
	\begin{equation*}
		\nabla_WJ
		=\nabla_{JV}J
		=-J\nabla_VJ
		=0.
	\end{equation*}
	Thus all four identities in \eqref{eq:torsion-reduction} hold, which finishes the proof of Proposition~\ref{prop:horizontal-torsion}.
\end{proof}

Next, by Proposition~\ref{prop:first-rigidity}, the Ricci tensor is $J$-invariant, and hence
\begin{equation*}
	Rc(JX,Y)=-Rc(X,JY).
\end{equation*}
Together with the symmetry of $Rc$, this identity shows that
\begin{equation}\label{eq:ricci-form}
	\rho(X,Y):=Rc(JX,Y)
\end{equation}
defines a real $J$-invariant two-form, called the \emph{Ricci form}. With this notation, we conclude the section by deriving identities needed in the next section and establishing the symmetry properties of $V$ and $W$.

\begin{lemma}\label{lem:soliton-symmetries}
	Let $(M^{2m},g,J,\omega,f)$, $m\geq2$, be a complete almost-K\"ahler gradient shrinking Ricci soliton. Then, for $V=\nabla f$ and $W=JV$, we have
	\begin{equation}\label{eq:beta-rho}
		\beta=2\lambda\omega-2\rho,
		\qquad
		\dd\rho=0,
	\end{equation}
	where $\beta=\Lie_V\omega$ and $\rho$ is the Ricci form defined in \eqref{eq:ricci-form}. Moreover,
	\begin{equation*}
		\Lie_VJ=0,
		\qquad
		\Lie_W\omega=\Lie_WJ=0.
	\end{equation*}
	Thus $W$ preserves the full almost-K\"ahler structure
	$(g,J,\omega)$.
\end{lemma}

\begin{proof}
	First, by Proposition~\ref{prop:first-rigidity}, $Rc$ is $J$-invariant and $\nabla_V\omega=0$. Using these facts and Lemma~\ref{lem:Hamilton}(i) in \eqref{eq:beta-decomp}, we obtain
	\begin{equation*}
		\beta
		=\frac{\Delta_gf}{m}\omega-2\rho_0
		=\frac{2m\lambda-R}{m}\omega
		-2\left(\rho-\frac{R}{2m}\omega\right)
		=2\lambda\omega-2\rho.
	\end{equation*}
	By \eqref{eq:beta}, $\beta$ is closed. Since $\omega$ is also closed, taking the exterior derivative gives $\dd\rho=0$, proving \eqref{eq:beta-rho}.
	
	Next, raising an index in the soliton equation \eqref{eq:soliton} gives
	\begin{equation*}
		(\Hess f)^\sharp=\lambda \operatorname{Id}-Rc^\sharp.
	\end{equation*}
	Since $Rc$ is $J$-invariant, $(\Hess f)^\sharp$ commutes with $J$. Moreover, by \eqref{eq:na-omega&na-J}, $\nabla_V\omega=0$ implies $\nabla_VJ=0$. Thus, for any tangent vector $X$,
	\begin{equation*}
		\begin{split}
			(\Lie_VJ)X
			&=(\nabla_VJ)X - \nabla_{JX}V + J\nabla_XV\\
			&=(\nabla_VJ)X
			-(\Hess f)^\sharp(JX)
			+J(\Hess f)^\sharp X\\
			&=0.
		\end{split}
	\end{equation*}
	Therefore, $\Lie_VJ=0$.
	
	Finally, since $\iota_W\omega=-\dd f$ and $\dd\omega=0$, Cartan's formula yields
	\begin{equation*}
		\Lie_W\omega
		=\dd (\iota_W\omega)+\iota_W\dd \omega
		=0.
	\end{equation*}
	Together with $\Lie_Wg=0$ from Proposition~\ref{prop:first-rigidity}, this implies
	\begin{equation*}
		0=(\Lie_W\omega)(X,Y)
		=(\Lie_Wg)(JX,Y)+g((\Lie_WJ)X,Y)
		=g((\Lie_WJ)X,Y)
	\end{equation*}
	for all tangent vectors $X,Y$. Hence $\Lie_WJ=0$, completing the proof.
\end{proof}

\section{The weighted Sekigawa identity and the proof of Theorem~\ref{thm:main}}
\label{sec:weighted-Sekigawa}

In this section, we combine the results of Sections~\ref{sec:coercive} and~\ref{sec:torsion} with the pointwise identity of Apostolov--Dr\u{a}ghici--Moroianu \cite[Proposition~2.1]{Apostolov-Draghici-Moroianu:01}, which recovers Sekigawa's formula \cite{Sekigawa:87} upon integration in the compact case. We recast the resulting identity in weighted divergence form and use it to prove Theorem~\ref{thm:main}.

Following \cite[Proposition~2.1]{Apostolov-Draghici-Moroianu:01} and \cite[equation~(12)]{Apostolov-Draghici:03}, we define
\begin{equation}\label{eq:phi-def}
	\phi(X,Y):=\langle\nabla_{JX}\omega,\nabla_Y\omega\rangle.
\end{equation}
By \eqref{eq:quasi-Kahler-identity}, $\phi$ is a real $J$-invariant two-form. Furthermore, it is semipositive since $\phi(X,JX)=|\nabla_{JX}\omega|^2\geq0$. For notational convenience, we also set
\begin{equation}\label{eq:chi-def}
	\chi:=\phi-\nabla^{\ast}\nabla\omega.
\end{equation}

\begin{lemma}\label{lem:chi-properties}
	Let $(M^{2m},g,J,\omega,f)$, $m\geq2$, be a complete almost-K\"ahler gradient shrinking Ricci soliton.
	Then the star-Ricci form $\rho^{\ast}$ and the Ricci form $\rho$ defined in \eqref{eq:rho*} and \eqref{eq:ricci-form}, respectively, satisfy
	\begin{equation}\label{eq:rho-star}
		\rho^{\ast}=\rho+\frac12\nabla^{\ast}\nabla\omega.
	\end{equation}
	Moreover, the two-form $\chi$ defined in \eqref{eq:chi-def} is closed and satisfies
	\begin{equation}\label{eq:chi-data}
		\Lambda\chi=\langle\chi,\omega\rangle=-\frac12|\nabla\omega|^2,
		\qquad
		\chi(V,W)=0,
	\end{equation}
	where $V=\nabla f$ and $W=JV$.
\end{lemma}

\begin{proof}
	First, by Proposition~\ref{prop:first-rigidity}, the Ricci tensor is $J$-invariant. Thus, with our conventions, the Hodge--Weitzenb\"ock formula \cite[\S~3.4, equations~(8)--(10)]{Apostolov-Draghici:03} reads
	\begin{equation*}
		\Delta_H\omega
		=\nabla^{\ast}\nabla\omega+2\rho-2\rho^\ast.
	\end{equation*}
	Since $\omega$ is harmonic, this gives \eqref{eq:rho-star}.
	
	Next, following \cite[\S~7]{Apostolov-Draghici:03}, we prove that $\chi$ is closed. Consider Gauduchon's first canonical Hermitian connection \cite{Gauduchon:97},
	\begin{equation*}
		\widetilde{\nabla}_XY
		=\nabla_XY-\frac12J(\nabla_XJ)Y,
	\end{equation*}
	and denote its Ricci form by $\tilde{\rho}$. As shown in \cite[\S~7]{Apostolov-Draghici:03}, Chern--Weil theory implies
	that $\tilde{\rho}$ is closed, and the following comparison formula holds:
	\begin{equation*}
		\tilde{\rho}=\rho^\ast-\frac12\phi.
	\end{equation*}
	By \eqref{eq:rho-star} and \eqref{eq:chi-def}, we obtain
	\begin{equation*}
		\begin{split}
			\tilde{\rho}
			&=\rho+\frac12\nabla^{\ast}\nabla\omega-\frac12\phi\\
			&=\rho-\frac12\chi.
		\end{split}
	\end{equation*}
	Since Lemma~\ref{lem:soliton-symmetries} gives $\dd\rho=0$, we obtain $\dd\chi=0$.
	
	It remains to verify \eqref{eq:chi-data}. Since $|\omega|^2=m$ is constant,
	\begin{equation*}
		0=\frac12\Delta_g|\omega|^2
		=-\langle\nabla^{\ast}\nabla\omega,\omega\rangle
		+|\nabla\omega|^2.
	\end{equation*}
	Consequently,
	\begin{equation}\label{eq:trace-na*na-omega}
		\Lambda(\nabla^{\ast}\nabla\omega)=\langle\nabla^{\ast}\nabla\omega,\omega\rangle=|\nabla\omega|^2.
	\end{equation}
	In an orthonormal $J$-adapted frame $\{e_a,Je_a\}_{a=1}^m$, \eqref{eq:phi-def} yields
	\begin{equation*}
		\Lambda\phi
		=\sum_{a=1}^m\phi(e_a,Je_a)
		=\sum_{a=1}^m|\nabla_{Je_a}\omega|^2
		=\frac12|\nabla\omega|^2,
	\end{equation*}
	where the last equality uses \eqref{eq:na-omega&na-J}, \eqref{eq:quasi-Kahler-identity}, and the orthogonality of $J$.
	It follows that
	\begin{equation*}
		\Lambda\chi
		=\Lambda\phi-\Lambda(\nabla^{\ast}\nabla\omega)
		=-\frac12|\nabla\omega|^2.
	\end{equation*}
	
	Finally, \eqref{eq:phi-def} and Proposition~\ref{prop:horizontal-torsion} imply
	\begin{equation*}
		\phi(V,W)=\langle\nabla_W\omega,\nabla_W\omega\rangle = 0.
	\end{equation*}
	Fix $p\in M$ and choose a local orthonormal frame $\{e_i\}_{i=1}^{2m}$ satisfying $\nabla_{e_i}e_j=0$ at $p$. By Proposition~\ref{prop:horizontal-torsion}, $(\nabla_{e_i}\omega)(V,W)\equiv0$ for each $i$. Differentiating in the $e_i$ direction and evaluating at $p$ gives
	\begin{equation*}
		0=(\nabla_{e_i}\nabla\omega)(e_i,V,W)
		+(\nabla_{e_i}\omega)(\nabla_{e_i}V,W)
		+(\nabla_{e_i}\omega)(V,\nabla_{e_i}W).
	\end{equation*}
	The last two terms vanish by Proposition~\ref{prop:horizontal-torsion}. Therefore, at $p$,
	\begin{equation*}
		(\nabla^{\ast}\nabla\omega)(V,W)
		=-\sum_{i=1}^{2m}
		(\nabla_{e_i}\nabla\omega)(e_i,V,W)
		=0.
	\end{equation*}
	Since $p$ is arbitrary, this and $\phi(V,W)=0$ imply $\chi(V,W)=0$, which finishes the proof of Lemma~\ref{lem:chi-properties}.
\end{proof}

Following \cite[pp.~776, 778]{Apostolov-Draghici-Moroianu:01}, let $P_J^-:\Lambda^2\to\Lambda_J^-$ denote the orthogonal projection, and let $\mathbb J$ be the natural complex structure on $\Lambda_J^-$, given by
\begin{equation*}
	(\mathbb J\alpha)(X,Y):=-\alpha(JX,Y),
	\qquad \alpha\in\Lambda_J^-.
\end{equation*}
We then define the projected curvature operator and its component anticommuting with $\mathbb J$ by
\begin{equation}\label{eq:R''-def}
	\widetilde{\mathcal R}
	:=P_J^-\circ\mathcal R\big|_{\Lambda_J^-},
	\qquad
	\widetilde{\mathcal R}^{\prime\prime}
	:=\frac12\left(
	\widetilde{\mathcal R}
	+\mathbb J\circ\widetilde{\mathcal R}\circ\mathbb J
	\right).
\end{equation}

We also define the one-form $\zeta$ by
\begin{equation}\label{eq:zeta-def}
	\zeta(X):=\langle\rho^{\ast},\nabla_X\omega\rangle.
\end{equation}
With the conventions fixed above, the identity of Apostolov--Dr\u{a}ghici--Moroianu \cite[Proposition~2.1, equation~(7)]{Apostolov-Draghici-Moroianu:01} becomes
\begin{equation}\label{eq:pointwise-Sekigawa}
	\Delta_H|\nabla\omega|^2
	=
	8\delta\zeta
	-8|\widetilde{\mathcal R}^{\prime\prime}|^2
	-|\nabla^{\ast}\nabla\omega|^2
	-|\phi|^2
	+4\langle\rho,\chi\rangle.
\end{equation}
Here we used $Rc^{\janti}=0$ from Proposition~\ref{prop:first-rigidity} and the definition of $\chi$ in \eqref{eq:chi-def}. In the compact case, integration of \eqref{eq:pointwise-Sekigawa} recovers Sekigawa's formula \cite[Proposition~3.2]{Sekigawa:87}.

We next express the Hodge Laplacian, codifferential, and Ricci terms in \eqref{eq:pointwise-Sekigawa} as weighted divergences, yielding the following pointwise identity.

\begin{lemma}
	\label{lem:weighted-Sekigawa-divergence}
	Let $(M^{2m},g,J,\omega,f)$, $m\geq2$, be a complete almost-K\"ahler
	gradient shrinking Ricci soliton. Then, for $V=\nabla f$ and $W=JV$, we have
	\begin{equation}\label{eq:density-Sekigawa}
		\begin{aligned}
			&\divergence\!\left\{
			e^{-f}\left[
			\nabla|\nabla\omega|^2
			+|\nabla\omega|^2V
			-8\zeta^\sharp
			+2\left(
			\ast^{-1}
			\frac{W^\flat\wedge\chi\wedge\omega^{m-2}}{(m-2)!}
			\right)^\sharp
			\right]\right\}\\
			&\qquad=
			e^{-f}\left(
			8|\widetilde{\mathcal R}^{\prime\prime}|^2
			+|\nabla^{\ast}\nabla\omega|^2
			+|\phi|^2
			+2\lambda|\nabla\omega|^2
			\right),
		\end{aligned}
	\end{equation}
	where $\zeta$, $\chi$, $\widetilde{\mathcal R}^{\prime\prime}$ and $\phi$ are defined in \eqref{eq:zeta-def}, \eqref{eq:chi-def}, \eqref{eq:R''-def}, and \eqref{eq:phi-def}, respectively.
\end{lemma}

\begin{proof}
	First, since $V=\nabla f$, the product rule yields
	\begin{equation}\label{eq:weighted-scalar-divergence}
		\divergence\!\left[
		e^{-f}\left(
		\nabla|\nabla\omega|^2+|\nabla\omega|^2V
		\right)\right]
		=
		e^{-f}\left[
		\Delta_g|\nabla\omega|^2
		+\bigl(\Delta_gf-|V|^2\bigr)|\nabla\omega|^2
		\right].
	\end{equation}
	
	Next, Proposition~\ref{prop:first-rigidity} gives
	$\nabla_V\omega=0$, and therefore
	\begin{equation*}
		\zeta(V)=\langle\rho^{\ast},\nabla_V\omega\rangle=0.
	\end{equation*}
	Since $\delta\zeta=-\divergence(\zeta^\sharp)$, it follows that
	\begin{equation}\label{eq:weighted-zeta-divergence}
		\divergence\!\left(-8e^{-f}\zeta^\sharp\right)
		=
		8e^{-f}\bigl(\delta\zeta+\zeta(V)\bigr)
		=
		8e^{-f}\delta\zeta.
	\end{equation}
	
	It remains to rewrite the Ricci term $4\langle\rho,\chi\rangle$. By \eqref{eq:beta-rho} and Lemma~\ref{lem:Hamilton}(i), $\beta$ is $J$-invariant and
	\begin{equation*}
		\Lambda\beta
		=2m\lambda-R
		=\Delta_gf.
	\end{equation*}
	Since $\Lambda\chi=-|\nabla\omega|^2/2$ by Lemma~\ref{lem:chi-properties}, the polarized form of Lemma~\ref{lem:Weil-two-forms} (see also \cite[Lemma~4.7]{Szekelyhidi:14}) implies
	\begin{equation}\label{eq:beta-chi-Weil}
		\frac{\beta\wedge\chi\wedge\omega^{m-2}}{(m-2)!}
		=
		\left[
		-\frac12(\Delta_gf)|\nabla\omega|^2
		-\langle\beta,\chi\rangle
		\right]\dd\mu_g.
	\end{equation}
	As in {\bf Part~I} of the proof of Lemma~\ref{lem:two-identities}, $\dd f=V^\flat$ and $W=JV$ imply that $\dd f\wedge W^\flat$ is $J$-invariant and satisfies
	\begin{equation*}
		\Lambda(\dd f\wedge W^\flat)=|V|^2,
		\qquad
		\langle\dd f\wedge W^\flat,\chi\rangle
		=\chi(V,W).
	\end{equation*}
	Applying the polarized form of Lemma~\ref{lem:Weil-two-forms} once more and using \eqref{eq:chi-data}, we obtain
	\begin{equation}\label{eq:df-W-chi-Weil}
		\frac{\dd f\wedge W^\flat\wedge\chi
			\wedge\omega^{m-2}}{(m-2)!}
		=
		-\frac12|V|^2|\nabla\omega|^2\dd\mu_g.
	\end{equation}
	Using $\beta=\dd W^\flat$, $\dd\chi=0$, and $\dd\omega=0$, together with \eqref{eq:beta-chi-Weil} and \eqref{eq:df-W-chi-Weil}, we have
	\begin{equation}\label{eq:exterior-diff}
		\dd\left(
		e^{-f}
		\frac{W^\flat\wedge\chi\wedge\omega^{m-2}}{(m-2)!}
		\right)
		=
		-e^{-f}\left[
		\frac12\bigl(\Delta_gf-|V|^2\bigr)|\nabla\omega|^2
		+\langle\beta,\chi\rangle
		\right]\dd\mu_g.
	\end{equation}
	
	For any $(2m-1)$-form $\Theta$,
	\begin{equation*}
		\dd\Theta
		=
		\divergence\!\left(
		(\ast^{-1}\Theta)^\sharp
		\right)\dd\mu_g;
	\end{equation*}
	see, e.g., \cite[p.~423, equation~(16.11) and the definition of divergence]{Lee:13}. Applying this identity to \eqref{eq:exterior-diff} yields
	\begin{equation}\label{eq:pointwise-beta-chi}
		\begin{aligned}
			e^{-f}\langle\beta,\chi\rangle
			={}&
			-\frac12e^{-f}
			\bigl(\Delta_gf-|V|^2\bigr)|\nabla\omega|^2\\
			&-\divergence\!\left[
			e^{-f}\left(
			\ast^{-1}
			\frac{W^\flat\wedge\chi\wedge\omega^{m-2}}{(m-2)!}
			\right)^\sharp
			\right].
		\end{aligned}
	\end{equation}
	Moreover, by Lemmas~\ref{lem:soliton-symmetries} and \ref{lem:chi-properties}, we have
	\begin{equation*}
		\rho=\lambda\omega-\frac12\beta,
		\qquad
		\Lambda\chi=-\frac12|\nabla\omega|^2.
	\end{equation*}
	Combining these identities with \eqref{eq:pointwise-beta-chi}, we obtain
	\begin{equation}\label{eq:weighted-rho-chi}
		\begin{aligned}
			4e^{-f}\langle\rho,\chi\rangle
			={}&
			e^{-f}\bigl(-2\lambda+\Delta_gf-|V|^2\bigr)
			|\nabla\omega|^2\\
			&+2\divergence\!\left[
			e^{-f}\left(
			\ast^{-1}
			\frac{W^\flat\wedge\chi\wedge\omega^{m-2}}{(m-2)!}
			\right)^\sharp
			\right].
		\end{aligned}
	\end{equation}
	
	Finally, multiplying \eqref{eq:pointwise-Sekigawa} by $e^{-f}$, using
	$\Delta_H=-\Delta_g$ on functions, and substituting
	\eqref{eq:weighted-rho-chi}, we obtain
	\begin{align*}
		&e^{-f}\left(
		8|\widetilde{\mathcal R}^{\prime\prime}|^2
		+|\nabla^{\ast}\nabla\omega|^2
		+|\phi|^2
		+2\lambda|\nabla\omega|^2
		\right)\\
		&\qquad=
		e^{-f}\left[
		\Delta_g|\nabla\omega|^2
		+\bigl(\Delta_gf-|V|^2\bigr)|\nabla\omega|^2
		+8\delta\zeta
		\right]\\
		&\qquad\quad
		+2\divergence\!\left[
		e^{-f}\left(
		\ast^{-1}
		\frac{W^\flat\wedge\chi\wedge\omega^{m-2}}{(m-2)!}
		\right)^\sharp
		\right].
	\end{align*}
	Together with \eqref{eq:weighted-scalar-divergence} and \eqref{eq:weighted-zeta-divergence}, this yields \eqref{eq:density-Sekigawa}, completing the proof.
\end{proof}

Now, integrating the identity in Lemma~\ref{lem:weighted-Sekigawa-divergence} against suitable cutoffs yields the following weighted Sekigawa formula.

\begin{proposition}[Weighted Sekigawa identity]
	\label{prop:weighted-Sekigawa}
	Let $(M^{2m},g,J,\omega,f)$, $m\geq2$, be a complete almost-K\"ahler
	gradient shrinking Ricci soliton. Then
	\begin{equation}\label{eq:weighted-Sekigawa-integral}
		\int_M e^{-f}\left(
		8|\widetilde{\mathcal R}^{\prime\prime}|^2
		+|\nabla^{\ast}\nabla\omega|^2
		+|\phi|^2
		+2\lambda|\nabla\omega|^2
		\right)\dd\mu_g=0,
	\end{equation}
	where $\widetilde{\mathcal R}^{\prime\prime}$ and $\phi$ are defined in \eqref{eq:R''-def} and \eqref{eq:phi-def}, respectively.
\end{proposition}

\begin{proof}
	Suppose first that $M$ is compact. Integrating
	\eqref{eq:density-Sekigawa} and applying Stokes' theorem immediately
	gives \eqref{eq:weighted-Sekigawa-integral}.
	
	Suppose now that $M$ is complete and noncompact. We choose the same cutoff function as in {\bf Part III} of the proof of Lemma~\ref{lem:two-identities},
	\begin{equation*}
		\psi_{\ell}:=\vartheta\left(\tfrac{f}{\ell}\right).
	\end{equation*}
	By a direct computation, we have
	\begin{align*}
		\Delta_g\psi_{\ell}-\langle V,\nabla \psi_{\ell}\rangle
		=\tfrac{1}{\ell}\vartheta'\left(\tfrac{f}{\ell}\right)
		\bigl(\Delta_gf-|V|^2\bigr)
		+\tfrac1{\ell^2}\vartheta''\left(\tfrac{f}{\ell}\right)|V|^2.
	\end{align*}
	On the transition region $\{\ell \leq f\leq2\ell \}\cap\{\psi_{\ell}>0\}$, the estimates in \eqref{eq:half-cutoff-bounds}, Lemma~\ref{lem:Hamilton}(i)\&(iii), and the defining properties of $\vartheta$ yield
	\begin{equation}\label{eq:Sekigawa-cutoff-bounds}
		\tfrac{|\nabla \psi_{\ell}|^2}{\psi_{\ell}}\leq\tfrac{C}{\ell},
		\qquad
		\left|
		\Delta_g\psi_{\ell}-\langle V,\nabla \psi_{\ell}\rangle
		\right|\leq C,
	\end{equation}
	where $C$ is independent of $\ell$. In the following, we write $\psi=\psi_{\ell}$ for simplicity.
	
	Multiplying \eqref{eq:density-Sekigawa} by $\psi^2$ and integrating by
	parts yields
	\begin{equation}\label{eq:localized-Sekigawa}
		\begin{split}
			&\int_M e^{-f}\psi^2\left(
			8|\widetilde{\mathcal R}^{\prime\prime}|^2
			+|\nabla^{\ast}\nabla\omega|^2
			+|\phi|^2+2\lambda|\nabla\omega|^2
			\right)\dd\mu_g\\
			&\qquad=8\int_M e^{-f}\zeta\bigl(\nabla(\psi^2)\bigr)\dd\mu_g
			-2\int_M e^{-f}
			\frac{\dd(\psi^2)\wedge W^\flat\wedge\chi
				\wedge\omega^{m-2}}{(m-2)!}\\
			&\qquad\quad+\int_M e^{-f}|\nabla\omega|^2\left(
			\Delta_g(\psi^2)
			-2\langle V,\nabla(\psi^2)\rangle
			\right)\dd\mu_g.
		\end{split}
	\end{equation}
	Here we used the Hodge star identity
	\begin{equation*}
		\left\langle\nabla(\psi^2),(\ast^{-1}\Theta)^\sharp\right\rangle
		\dd\mu_g
		=\dd(\psi^2)\wedge\ast(\ast^{-1}\Theta)
		=\dd(\psi^2)\wedge\Theta,
	\end{equation*}
	for any $(2m-1)$-form $\Theta$. Since $\nabla(\psi^2)$ is parallel to $V$ and $\zeta(V)=\langle\rho^{\ast},\nabla_V\omega\rangle=0$ by Proposition~\ref{prop:first-rigidity}, the first term on the right-hand side of \eqref{eq:localized-Sekigawa} vanishes.
	
	Furthermore, since $\dd(\psi^2)\wedge W^\flat$ is $J$-invariant, the same computation leading to \eqref{eq:df-W-chi-Weil} in Lemma~\ref{lem:weighted-Sekigawa-divergence}, together with \eqref{eq:chi-data}, gives
	\begin{equation*}
		\frac{\dd(\psi^2)\wedge W^\flat\wedge\chi
			\wedge\omega^{m-2}}{(m-2)!}
		=-\frac12\langle V,\nabla(\psi^2)\rangle
		|\nabla\omega|^2\dd\mu_g.
	\end{equation*}
	Thus \eqref{eq:localized-Sekigawa} reduces to
	\begin{equation}\label{eq:localized-Sekigawa-reduced}
		\begin{split}
			&\int_M e^{-f}\psi^2\left(
			8|\widetilde{\mathcal R}^{\prime\prime}|^2
			+|\nabla^{\ast}\nabla\omega|^2
			+|\phi|^2+2\lambda|\nabla\omega|^2
			\right)\dd\mu_g\\
			&\qquad=
			\int_M e^{-f}|\nabla\omega|^2
			\left(
			\Delta_g(\psi^2)
			-\langle V,\nabla(\psi^2)\rangle
			\right)\dd\mu_g.
		\end{split}
	\end{equation}
	
	Next, by \eqref{eq:trace-na*na-omega}, we have $|\nabla\omega|^2=\langle\nabla^{\ast}\nabla\omega,\omega\rangle$. Hence, by Cauchy--Schwarz and $|\omega|^2=m$,
	\begin{equation*}
		|\nabla\omega|^4
		\leq m|\nabla^{\ast}\nabla\omega|^2.
	\end{equation*}
	Then for any $\varepsilon>0$, by Young's inequality, we have
	\begin{equation}\label{eq:epsilon-na*na-omega}
		\begin{split}
			&\left|
			\int_M e^{-f}|\nabla\omega|^2
			\left(
			\Delta_g(\psi^2)-\langle V,\nabla(\psi^2)\rangle
			\right)\dd\mu_g
			\right|\\
			&\qquad\leq
			\varepsilon\int_M e^{-f}\psi^2
			|\nabla^{\ast}\nabla\omega|^2\dd\mu_g\\
			&\qquad\quad
			+C_{m,\varepsilon}
			\int_{\{\psi>0\}}e^{-f}
			\frac{\left|
			\Delta_g(\psi^2)-\langle V,\nabla(\psi^2)\rangle
			\right|^2}{\psi^2}\dd\mu_g.
		\end{split}
	\end{equation}
	Note that, by direct computation,
	\begin{equation*}
		\Delta_g(\psi^2)-\langle V,\nabla(\psi^2)\rangle
		=
		2\psi\left(
		\Delta_g\psi-\langle V,\nabla \psi\rangle
		\right)+2|\nabla \psi|^2.
	\end{equation*}
	Therefore, on the transition region $\{\ell\leq f\leq 2\ell\}\cap\{\psi>0\}$, \eqref{eq:Sekigawa-cutoff-bounds} implies
	\begin{equation*}
		\begin{aligned}
			\frac{\left|
				\Delta_g(\psi^2)-\langle V,\nabla(\psi^2)\rangle
				\right|^2}{\psi^2}
			\leq
			8\left|
			\Delta_g\psi-\langle V,\nabla \psi\rangle
			\right|^2
			+8\frac{|\nabla \psi|^4}{\psi^2}
			\leq C.
		\end{aligned}
	\end{equation*}
	Choosing $\varepsilon>0$ sufficiently small, we absorb the first term on the right-hand side of \eqref{eq:epsilon-na*na-omega} into the left-hand side of \eqref{eq:localized-Sekigawa-reduced} and obtain
	\begin{equation*}
		\int_M e^{-f}\psi^2\left(
		8|\widetilde{\mathcal R}^{\prime\prime}|^2
		+|\nabla^{\ast}\nabla\omega|^2
		+|\phi|^2+2\lambda|\nabla\omega|^2
		\right)\dd\mu_g
		\leq
		C\int_{\{\ell\leq f\leq2\ell\}}e^{-f}\dd\mu_g.
	\end{equation*}
	The right-hand side tends to zero as $\ell \rightarrow \infty$ by Lemma~\ref{lem:Cao-Zhou}.
	
	Since $\lambda>0$, the integrand below is nonnegative. For every relatively compact domain $L\Subset M$, we have $\psi\equiv1$ on $L$ for all sufficiently large $\ell$. Hence
	\begin{align*}
		&\int_L e^{-f}\left(
		8|\widetilde{\mathcal R}^{\prime\prime}|^2
		+|\nabla^{\ast}\nabla\omega|^2
		+|\phi|^2+2\lambda|\nabla\omega|^2
		\right)\dd\mu_g\\
		&\qquad\leq
		\int_M e^{-f}\psi^2\left(
		8|\widetilde{\mathcal R}^{\prime\prime}|^2
		+|\nabla^{\ast}\nabla\omega|^2
		+|\phi|^2+2\lambda|\nabla\omega|^2
		\right)\dd\mu_g
		\longrightarrow0.
	\end{align*}
	As $L$ is arbitrary, \eqref{eq:weighted-Sekigawa-integral} follows;
	it also follows directly from Fatou's lemma. This completes the proof of Proposition~\ref{prop:weighted-Sekigawa}.
\end{proof}

Finally, we conclude the proof of Theorem~\ref{thm:main}.

\begin{proof}[\bf Proof of Theorem~\ref{thm:main}]
	Since the Ricci soliton is shrinking, we have $\lambda>0$. Therefore, for $m\geq2$, Proposition~\ref{prop:weighted-Sekigawa} yields $\nabla\omega=0$. Hence $J$ is parallel, and the almost-K\"ahler structure is K\"ahler. This completes the proof of Theorem~\ref{thm:main}.
\end{proof}


\end{document}